\documentclass[reqno,12pt]{amsart}
\theoremstyle{plain}
\newtheorem{thm}{Theorem}
\newtheorem{lemma}[thm]{Lemma} 
\newtheorem{prop}[thm]{Proposition}

\theoremstyle{remark}

\theoremstyle{definition}

\usepackage{amscd,amssymb,comment,epic,eepic,euscript,graphics}
\usepackage{enumerate}
\usepackage[initials]{amsrefs}
\usepackage[all]{xy}
\usepackage{color}

\newcommand\Afr{{\mathfrak A}}

\newcommand\Dc{{\mathcal{D}}}

\newcommand\Exp{{\operatorname{Exp}}}
\newcommand\HEu{{\EuScript H}}
\newcommand\lspan{\mathrm{span}\,}
\newcommand\Mcal{{\mathcal{M}}}

\newcommand\supp{\operatorname{supp}}

\begin{document}

\title[Upper triangular forms]{Holomorphic functional calculus on upper triangular forms in finite von Neumann algebras}

\author[Dykema]{K. Dykema$^*$}
\address{Department of Mathematics, Texas A\&M University, College Station, TX, USA.}
\email{ken.dykema@math.tamu.edu}
\thanks{\footnotesize ${}^{*}$ Research supported in part by NSF grant DMS--1202660.}
\author[Sukochev]{F. Sukochev$^{\S}$}
\address{School of Mathematics and Statistics, University of new South Wales, Kensington, NSW, Australia.}
\email{f.sukochev@math.unsw.edu.au}
\thanks{\footnotesize ${}^{\S}$ Research supported by ARC}
\author[Zanin]{D. Zanin$^{\S}$}
\address{School of Mathematics and Statistics, University of new South Wales, Kensington, NSW, Australia.}
\email{d.zanin@math.unsw.edu.au}



\begin{abstract} 
The decompositions of an element of a finite von Neumann algebra into the sum of a normal operator plus an s.o.t.-quasinilpotent operator, obtained using the Haagerup--Schultz hyperinvariant projections, behave well with respect to holomorphic functional calculus.
\end{abstract}

\maketitle

\section{Introduction and description of results}

This note concerns the decomposition theorem for elements of a finite von Neumann algebra,
recently proved in~\cite{DSZ}.
In that paper, given a von Neumann algebra $\Mcal$ with a normal, faithful, tracial state $\tau$,
by using the hyperinvariant subspaces found by Haagerup and Schultz~\cite{HS09}
and their behavior with respect to Brown measure,
for every element $T\in\Mcal$
we constructed a decomposition $T=N+Q$ where $N\in\Mcal$ is a normal operator whose
Brown measure agrees with that of $T$ and where $Q$ is an s.o.t.-quasinilpotent operator.
An element $Q\in\Mcal$ is said to be s.o.t.-quasinilpotent if
$((Q^*)^nQ^n)^{1/n}$ converges in the strong operator topology to the zero operator
---  by Corollary 2.7 in \cite{HS09}, this is equivalent to  the Brown measure of $Q$ being concentrated at $0$.
In fact, $N$ is obtained as the conditional expectation of $T$ onto the (abelian)
subalgebra generated by an increasing family of Haagerup--Schultz projections.

The Brown measure~\cite{Br86} of an element $T$ of a finite von Neumann algebra is a sort of spectral distribution measure,
whose support is contained in the spectrum $\sigma(T)$ of $T$.
We will use $\nu_T$ to denote the Brown measure of $T$.
The Brown measure behaves well under holomorphic (or Riesz) functional calculus.
Indeed, Brown proved (Theorem~4.1 of~\cite{Br86}) that if $h$ is holomorphic on a neighborhood of the spectrum of $T$,
then $\nu_{h(T)}=\nu_T\circ h^{-1}$ (the push-forward measure by the function $h$).

In this note, we prove the following:
\begin{thm}\label{thm:main}
Let $T$ be an element of a finite von Neumann algebra $\Mcal$ (with fixed normal, faithful tracial state $\tau$)
and let
$T=N+Q$ be a decomposition from~\cite{DSZ}, with $N$ normal, $\nu_N=\nu_T$ and $Q$ s.o.t.-quasinilpotent.
\begin{enumerate}[(i)]
\item
Let $h$ be a complex-valued function that is holomorphic
on a neighborhood of the spectrum of $T$.
Then
\[
h(T)=h(N)+Q_h,
\]
where $Q_h$ is s.o.t.-quasinilpotent.
\item
If $0\notin\supp\nu_T$ (so that $N$ is invertible), then
\[
T=N(I+N^{-1}Q)
\]
and $N^{-1}Q$ is s.o.t.-quasinilpotent.
\end{enumerate}
\end{thm}
The key result for the proof is Lemma~22 of~\cite{DSZ}, which allows us to reduce to the case when $N$ and $Q$ commute.
Before using this, we require a few easy results about s.o.t.-quasinilpotent operators on Hilbert space.

\begin{lemma}\label{lem:invs}
Let $\Afr$ be a unital algebra and let $N,Q\in\Afr$, $T=N+Q$ and suppose that both $N$ and $T$ are invertible.
Then
\[
T^{-1}=N^{-1}-T^{-1}QN^{-1}.
\]
\end{lemma}
\begin{proof}
We have
\[
T^{-1}-N^{-1}=T^{-1}(N-T)N^{-1}=-T^{-1}QN^{-1}.
\]
\end{proof}

\begin{lemma}\label{lem:AQ}
Let $A$ and $Q$ be bounded operators on a Hilbert space $\HEu$ such that $AQ=QA$
and suppose $Q$ is s.o.t.-quasinilpotent.
Then $AQ$ is s.o.t.-quasinilpotent.
\end{lemma}
\begin{proof}
We have $(AQ)^n=A^nQ^n$ and
\[
((AQ)^*)^n(AQ)^n=(Q^*)^n(A^*)^nA^nQ^n\le\|A\|^{2n}(Q^*)^nQ^n.
\]
By Loewner's Theorem, for $n\ge2$ the function $t\mapsto t^{2/n}$ is operator monotone and we have
\[
\big(((AQ)^*)^n(AQ)^n\big)^{2/n}\le\|A\|^4\big((Q^*)^nQ^n\big)^{2/n}.
\]
Thus, for $\xi\in\HEu$, we have
\begin{multline*}
\|\big(((AQ)^*)^n(AQ)^n\big)^{1/n}\xi\|^2
=\langle \big(((AQ)^*)^n(AQ)^n\big)^{2/n}\xi,\xi\rangle \\
\le\|A\|^4\langle\big((Q^*)^nQ^n\big)^{2/n}\xi,\xi\rangle
=\|A\|^4\|\big((Q^*)^nQ^n\big)^{1/n}\xi\|^2.
\end{multline*}
Since $Q$ is s.o.t.-quasinilpotent, this tends to zero as $n\to\infty$.
\end{proof}

\begin{prop}\label{prop:hTNcomm}
Let $N$ and $Q$ be bounded operators on a Hilbert space and suppose $NQ=QN$ and $Q$ is s.o.t.-quasinilpotent.
Let $T=N+Q$.
Let $h$ be a function that is holomorphic on a neighborhood of
the union $\sigma(T)\cup\sigma(N)$ of the spectra of $T$ and $N$.
Then $h(T)$ and $h(N)$ commute, and $h(T)-h(N)$ is s.o.t.-quasinilpotent.
\end{prop}
\begin{proof}.
If $\lambda$ is outside of $\sigma(T)\cup\sigma(N)$, then by Lemma~\ref{lem:invs},
\begin{equation}\label{eq:TNlam}
(T-\lambda)^{-1}=(N-\lambda)^{-1}-(T-\lambda)^{-1}Q(N-\lambda)^{-1}.
\end{equation}
Let $C$ be a contour in the domain of the complement $\sigma(T)\cup\sigma(N)$, with winding number $1$
around each point in $\sigma(T)\cup\sigma(N)$.
Then
\begin{align*}
h(T)&=\frac1{2\pi i}\int_C h(\lambda)(\lambda-T)^{-1}\,d\lambda \\
h(N)&=\frac1{2\pi i}\int_C h(\lambda)(\lambda-N)^{-1}\,d\lambda.
\end{align*}
For any complex numbers $\lambda_1$ and $\lambda_2$ outside of $\sigma(T)\cup\sigma(N)$, the operators $(\lambda_1-T)^{-1}$, $(\lambda_2-N)^{-1}$ and $Q$
commute;
thus, $h(T)$ and $h(N)$ commute with each other.
Using~\eqref{eq:TNlam}, we have
\[
h(T)-h(N)=\frac1{2\pi i}\int_C h(\lambda)(\lambda-T)^{-1}Q(\lambda-N)^{-1}\,d\lambda=AQ,
\]
where
\[
A=\frac1{2\pi i}\int_C h(\lambda)(\lambda-T)^{-1}(\lambda-N)^{-1}\,d\lambda.
\]
We have $AQ=QA$.
By Lemma~\ref{lem:AQ},
$AQ$ is s.o.t.-quasinilpotent.
\end{proof}

For the remainder of this note, $\Mcal$ will be a finite von Neumann algebra
with specified normal, faithful, tracial state $\tau$.

\begin{lemma}\label{lem:hTinv}
Let $T\in\Mcal$.
Suppose $p\in\Mcal$ is a $T$--invariant projection with $p\notin\{0,1\}$.
\begin{enumerate}[(i)]
\item
If $T$ is invertible, then $p$ is $T^{-1}$--invariant.  Moreover, we have
\begin{align*}
T^{-1}p&=(pTp)^{-1}, \\
(1-p)T^{-1}&=((1-p)T(1-p))^{-1},
\end{align*}
where the inverses on the right-hand-sides are in $p\Mcal p$ and $(1-p)\Mcal(1-p)$, respectively.
\item
The union of the spectra of $pTp$ and $(1-p)T(1-p)$ (in $p\Mcal p$ and $(1-p)\Mcal(1-p)$, respectively)
equals the spectrum of $T$.
\item
If $h$ is a function that is holomorphic
on a neighborhood of $\sigma(T)$, then $p$ is $h(T)$--invariant.
Moreover, $h(T)p=h(pTp)$.
\end{enumerate}
\end{lemma}
\begin{proof}
For~(i), a key fact is that one-sided invertible elements of $\Mcal$ are always invertible.
Thus, writing $T=\left(\begin{smallmatrix} a&b\\0&c\end{smallmatrix}\right)$ with respect to the projections
$p$ and $(1-p)$ (so that $a=pTp$, $b=pT(1-p)$ and $c=(1-p)T(1-p)$)
writing $T^{-1}=\left(\begin{smallmatrix} x&y\\w&z\end{smallmatrix}\right)$ and multiplying,
we easily see that $a$ and $c$ must be invertible
and
\begin{equation}\label{eq:Tinv}
T^{-1}=\left(\begin{matrix}a^{-1}&-a^{-1}bc^{-1}\\0&c^{-1}\end{matrix}\right).
\end{equation}
Thus, $p$ is $T^{-1}$--invariant.

For~(ii) we use~(i) and the fact that the formula~\eqref{eq:Tinv} shows that
$T$ is invertible whenever $pTp$ and $(1-p)T(1-p)$ are invertible.

For~(iii), writing
\begin{equation}\label{eq:hT}
h(T)=\frac1{2\pi i}\int_Ch(\lambda)(\lambda-T)^{-1}\,d\lambda
\end{equation}
for a suitable contour $C$, where this is a Riemann integral that converges in norm,
the result follows by applying part~(i).
\end{proof}

For a von Neumann subalgebra $\Dc$ of $\Mcal$, let
$\Exp_\Dc$ and
 $\Exp_{\Dc'}$, respectively denote the $\tau$-preserving conditional expectations
onto $\Dc$ and, respectively, the relative commutant of $\Dc$ in $\Mcal$.

\begin{lemma}\label{lem:Einv}
Let $T\in\Mcal$.
\begin{enumerate}[(i)]
\item
Suppose $0=p_0\le p_1\le\cdots\le p_n=1$ are $T$-invariant projections and let $\Dc=\lspan\{p_1,\ldots,p_n\}$.
Then the spectra of $T$ and of $\Exp_{\Dc'}(T)$ agree.
If $T$ is invertible, then $\Exp_{\Dc'}(T^{-1})=\Exp_{\Dc'}(T)^{-1}$.
\item
Suppose $(p_t)_{0\le t\le 1}$ is an increasing family of $T$-invariant projections in $\Mcal$ with $p_0=0$
and $p_1=1$, that is right-continuous with respect to strong operator topology.
Let $\Dc$ be the von Neumann algebra generated by the set of all $p_t$.
If $T$ is invertible, then so is $\Exp_{\Dc'}(T)$ and $\Exp_{\Dc'}(T^{-1})=\Exp_{\Dc'}(T)^{-1}$.
\end{enumerate}
\end{lemma}
\begin{proof}
For~(i), we have
\[
\Exp_{\Dc'}(T)=\sum_{j=1}^n(p_j-p_{j-1})T(p_j-p_{j-1}).
\]
The assertions now follow from repeated application of Lemma~\ref{lem:hTinv}.

For~(ii), using the right-continuity of $p_t$ it is easy to choose an increasing family of finite dimensional subalgebras
$\Dc_n$ of $\Dc$ whose union is strong operator topology dense in $\Dc$.
Then $\Exp_{\Dc_n'}(T)$ and $\Exp_{\Dc_n'}(T^{-1})$ converge in strong operator topology to
$\Exp_{\Dc'}(T)$ and $\Exp_{\Dc'}(T^{-1})$, respectively, and both sequences are bounded.
From~(i) we have the equality
\[
\Exp_{\Dc_n'}(T)\Exp_{\Dc_n'}(T^{-1})=I,
\]
and taking the limit as $n\to\infty$
yields the desired result.
\end{proof}

\begin{lemma}\label{lem:hTexp}
Let $T\in\Mcal$ and let $p_t$ and $\Dc$ be as in either part~(i) or part~(ii) of  Lemma~\ref{lem:Einv}.
Suppose a function $h$ is holomorphic on a neighborhood of the spectrum of $T$.
Then $\Exp_{\Dc'}(h(T))=h(\Exp_{\Dc'}(T))$.
\end{lemma}
\begin{proof}
Using that the Riemann integral~\eqref{eq:hT} converges in norm, that $\Exp_{\Dc'}$ is norm continuous
and applying Lemma~\ref{lem:Einv}, we get
\begin{multline*}
\Exp_{\Dc'}(h(T))=\frac1{2\pi i}\int_Ch(\lambda)\Exp_{\Dc'}\big((\lambda-T)^{-1}\big)\,d\lambda \\
=\frac1{2\pi i}\int_Ch(\lambda)(\lambda-\Exp_{\Dc'}(T))^{-1}\,d\lambda=h(\Exp_{\Dc'}(T)).
\end{multline*}
\end{proof}

For convenience, here is the statement of Lemma~22 of~\cite{DSZ}
and an immediate consequence.
\begin{lemma}\label{lem:22}
Let $T\in\Mcal$.
For any increasing, right-continuous family of $T$-invariant projections $(q_t)_{0\le t\le 1}$ with $q_0=0$ and $q_1=1$,
letting $\Dc$ be the von Neumann algebra generated by the set of all the $q_t$,
the Fuglede--Kadison determinants of $T$ and $\Exp_{\Dc'}(T)$ agree.
Since the same is true for $T-\lambda$ and $\Exp_{\Dc'}(T)-\lambda$ 
for all complex numbers $\lambda$,
we have that the Brown measures of $T$ and $\Exp_{\Dc'}(T)$ agree.
\end{lemma}

Now we have all the ingredients to prove our main result.
\begin{proof}[Proof of Theorem~\ref{thm:main}]
In Theorem 6 of \cite{DSZ} the decomposition $T=N+Q$ is constructed by considering an increasing, right-continuous
family $(p_t)_{0\le t\le1}$
of Haagerup--Schultz projections, with $p_0=0$ and $p_1=1$,
that are $T$-invariant, letting $\Dc$ be the von Neumann algebra generated
by the set of projections in this family
and taking $N=\Exp_\Dc(T)$.
In particular, each $p_t$ is also $Q$-invariant.

For~(i), we need to show that the Brown measure of $h(T)-h(N)$ is the Dirac mass at $0$.
By Lemma~\ref{lem:hTinv}(iii), each $p_t$ is $h(T)$-invariant.
So by Lemma~\ref{lem:22}, the Brown measures of
$h(T)-h(N)$ and $\Exp_{\Dc'}(h(T)-h(N))$ agree.
Since $h(N)\in\Dc$, we have $\Exp_{\Dc'}(h(N))=h(N)$
and by Lemma~\ref{lem:hTexp}, we have $\Exp_{\Dc'}(h(T))=h(\Exp_{\Dc'}(T))$.
Combining these facts we get
\begin{equation}\label{55}
\nu_{h(T)-h(N)}=\nu_{h(\Exp_{\Dc'}(T))-h(N)}.
\end{equation}
We have
\[
\Exp_{\Dc'}(T)=N+\Exp_{\Dc'}(Q)
\]
and $\Exp_{\Dc'}(Q)$ is s.o.t.-quasinilpotent.
This last statement follows formally from Lemma~\ref{lem:22} and the fact that $Q$ is s.o.t.-quasinilpotent.
However, we should mention that the fact that $\Exp_{\Dc'}(Q)$ is s.o.t.-quasinilpotent was actually proved directly in~\cite{DSZ} as a step in the proof that $Q$ is s.o.t.-quasinilpotent.
In any case, since $N$ and $\Exp_{\Dc'}(T)$ commute and $\Exp_{\Dc'}(Q)$ is s.o.t.-quasinilpotent, by Proposition~\ref{prop:hTNcomm} it follows that
$h(\Exp_{\Dc'}(T))-h(N)$ is s.o.t.-quasinilpotent.
Using \eqref{55}, we get that $h(T)-h(N)$ is s.o.t.-quasinilpotent, as desired.

For~(ii), the projections $p_t$ form a right-continuous family, each of which is invariant under $N^{-1}Q$.
By Lemma~\ref{lem:22}, the Brown measure of $N^{-1}Q$ equals the Brown measure of
\begin{equation}\label{eq:EQ}
\Exp_{\Dc'}(N^{-1}Q)=N^{-1}\Exp_{\Dc'}(Q).
\end{equation}
But since $N^{-1}$ and $\Exp_{\Dc'}(Q)$ commute and since the latter is s.o.t.-quasinilpotent,
by Lemma~\ref{lem:AQ}, their product~\eqref{eq:EQ} is s.o.t.-quasinilpotent.
\end{proof}

\end{document}